\documentclass{article}
\usepackage[paper=a4paper,left=30mm,right=30mm,top=25mm,bottom=25mm]{geometry}
\usepackage{graphicx}
\usepackage{subcaption}
\usepackage{lmodern}
\usepackage{comment}
\usepackage{dsfont}
\usepackage{pifont}
\usepackage{amsmath}
\usepackage{amsthm}
\usepackage{amssymb}
\usepackage{booktabs}
\usepackage{mathtools}
\usepackage{enumerate}
\usepackage{color}
\usepackage{csquotes}
\usepackage[numbers]{natbib}
\usepackage{hyperref}
  
\newtheorem{theorem}{Theorem}[section]
\newtheorem{lemma}[theorem]{Lemma}
\newtheorem{corollary}[theorem]{Corollary}
\newtheorem{proposition}[theorem]{Proposition}

\newtheorem{remark}[theorem]{Remark}

\newcommand{\1}{\mathds{1}}
\makeatletter
\newcommand\nopagebreakhere{\par\nobreak\@afterheading}
\makeatother

\newcommand{\de}{\mathrm{\,d}}
\newcommand{\CC}{\mathcal{C}}  % set of copulas
\newcommand{\R}{\mathbb{R}}
\providecommand{\keywords}[1]{\textbf{Keywords } #1}
\newcommand{\PiC}{\Pi}         % Independence copula
\newcommand{\MC}{M}            % Comonotonicity copula
\newcommand{\WC}{W}            % Countermonotonicity copula

\title{\textbf{On the exact region between Chatterjee's rank correlation and Spearman's footrule}}
\author{Marcus Rockel}
\date{\today}

\begin{document}

\maketitle
\begin{center}
	\small\textit{
		Department of Quantitative Finance,\\
		Institute for Economics, University of Freiburg,\\
		Rempartstr.\;16, 79098 Freiburg, Germany,\\
        \texttt{marcus.rockel@finance.uni-freiburg.de} \\[2mm]
	}
\end{center}
\begin{abstract}
Chatterjee's rank correlation \(\xi\) has emerged as a popular measure quantifying the strength of directed functional dependence between random variables $X$ and $Y$.
If $X$ and $Y$ are continuous, $\xi$ equals Spearman's footrule~\(\psi\) for the Markov product of the copula induced by $(X,Y)$ and its transpose.
We analyze the relationship between these two measures more in depth by studying the attainable region of possible pairs \((\xi, \psi)\) over all bivariate copulas.
In particular, we show that for given $\xi$, the maximal possible value of $\psi$ is uniquely attained by a Fréchet copula.
As a by-product of this and a known result for Markov products of copulas, we obtain that \(\xi\le\psi\le \sqrt{\xi}\) characterizes the exact region of stochastically increasing copulas.
Regarding the minimal possible value of \(\psi\) for given \(\xi\), we give a lower bound based on Jensen's inequality and construct a two-parameter copula family that comes comparably close.
\end{abstract} 
\vspace{0ex}

\section{Introduction}
\label{sec:intro}

Let $\CC$ denote the space of all bivariate \emph{copulas}, i.e.\;the set of joint cumulative distribution functions defined on the unit square $[0,1]^2$ with standard uniform marginals.
If $X$ and $Y$ are continuous random variables, then by Sklar's theorem there exists a unique copula $C\in\CC$ that satisfies
\(
    F_{X,Y}(x,y) = C(F_X(x), F_Y(y))
\)
for all \(x,y \in \mathbb{R}\), where \(F_{X,Y}\), \(F_X\), and \(F_Y\) denote the joint and marginal cumulative distribution functions of \(X\) and \(Y\), respectively, see \cite{sklar1959fonctions}.
Two important measures of dependence between continuous random variables $X$ and $Y$ that can be written purely as functions of the associated copula are \textit{Chatterjee's rank correlation} \(\xi\) and \textit{Spearman's footrule}~\(\psi\).
Chatterjee's rank correlation \(\xi:\CC\rightarrow[0,1]\), which is also known as the \emph{Dette-Siburg-Stoimenov dependence measure}, measures the degree of functional dependence and is given by
\begin{align}\label{eq:xi-definition}
\xi(C) \coloneq 6 \int_0^1\int_0^1 \left(\partial_1 C(u,v)\right)^2 \de u \de v - 2,
\end{align}
see \cite{chatterjee2020,dette2013copula}.
Note that $\xi$ is certainly well-defined since copulas are almost everywhere partially differentiable, cf.~\cite[Thm.~2.2.7]{Nelsen-2006}.
It holds that $\xi(C)=0$ if and only if $X$ and $Y$ are independent, and $\xi(C)=1$ if and only if $Y$ is a measurable function of $X$, almost surely. 
Generally, the closer \(\xi(C)\) is to \(1\), the stronger the functional dependence of \(Y\) on \(X\), but an exact interpretation of intermediate values remains challenging.

Spearman's footrule \(\psi:\CC\rightarrow[-\frac12,1]\) on the other hand is a measure of association used e.g.\;for comparing sets of ranks, and is defined as
\begin{align}\label{eq:psi-definition}
\psi(C) \coloneq 6 \int_0^1 C(u,u) \de u - 2,
\end{align}
see \cite{genest2010spearman} for an overview.
Let \(C^\top(u,v) := C(v,u)\) for $u,v\in[0,1]$ be the \emph{transpose} of $C$ and define for $C_1,C_2\in\CC$ the \emph{Markov product} \((C_1 \ast C_2)(u,v) \coloneq \int_0^1 \partial_2 C_1(u,t)\partial_1 C_2(t,v) \de t\), $(u,v)\in[0,1]^2$, which is again a copula by \cite[Thm.~2.1]{darsow1992copulas}.
%Intuitively, the Markov product of two copulas describes the dependence structure obtained by ``gluing'' them along their common marginal, see \cite{olsen1996copulas} for more details.
With this definition in hand, one can rewrite \(\xi(C)\) as
\begin{align}\label{eq:xi-psi-relation}
\xi(C)
= 6\int_0^1 \int_0^1 \partial_2 C^\top(v,u) \partial_1 C(u,v) \de u \de v - 2
= 6\int_0^1 (C^\top\ast C)(v,v) \de v - 2
= \psi(C^\top \ast C).
\end{align}
This is an insightful way of writing \(\xi(C)\), as Spearman's footrule is a classical measure of association that has been extensively studied in the past decades.
In particular, several works have analyzed exact regions between Spearman's footrule and other measures of association over the full space $\CC$, see \cite{kokol2022exact,bukovvsek2023exact,bukovvsek2024exact,kokol2021exact,tschimpke2025revisiting}.
In comparison, the literature on Chatterjee's rank correlation is still in its infancy and concerning exact regions with other measures of association over the full space $\CC$, to the best of our knowledge, only the exact region with Spearman's rho $\rho$ has been established so far, see \cite{ansari2025exact}.
The $(\xi,\rho)$-region provides sharp upper and lower bounds for the possible values of $\rho$ given $\xi$.
For example, it was shown that when $\xi=0.3$ the absolute value of $\rho$ cannot exceed $0.7$.
%A fundamental relationship between Chatterjee's rank correlation and Spearman's footrule is as follows.

Motivated by these previous results and the natural relation \eqref{eq:xi-psi-relation} between \(\xi\) and \(\psi\), we study the exact region of possible pairs \((\xi(C), \psi(C))\) over all copulas \(C \in \CC\).
To do so, we express \(\psi(C)\) in terms of \(h(u,v) \coloneqq \partial_1 C(u,v)\).
Since \(C(u,u) = \int_0^u h(t,u) \de t\), one can write
\begin{align}\label{eq:psi-in-terms-of-h}
\psi(C) = 6 \int_0^1 \left( \int_0^u h(t,u) \de t \right) \de u - 2
= 6 \int_0^1 \int_0^1 \1_{\{t \le v\}} h(t,v) \de t \de v - 2.
\end{align}
\begin{figure}[t!]
\centering
\includegraphics[width=\textwidth]{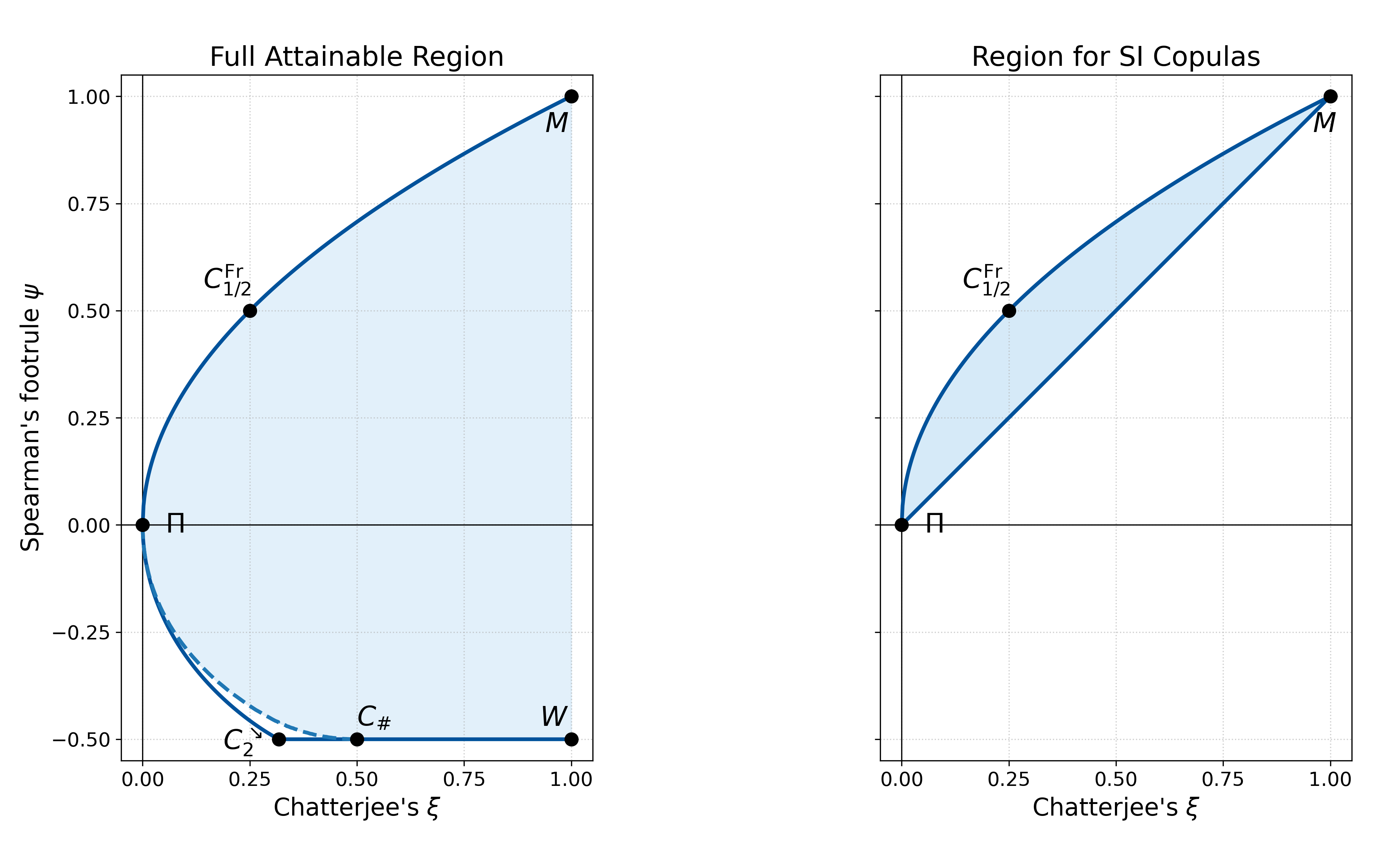}
\caption{
    The attainable $(\xi,\psi)$-region for all copulas on the left and the $(\xi,\psi)$-region for stochastically increasing copulas on the right.
    $\PiC(u,v):=uv$, $\MC(u,v):=\min\{u,v\}$ and $\WC(u,v):=\max\{u+v-1, 0\}$, $u,v\in[0,1]$, denote the independence, upper and lower Fréchet copulas, respectively, $C_{\#}$ the $2\times 2$-checkerboard copula with zero mass on the main diagonal, which minimizes $\xi$ over all copulas with $\psi = -0.5$, and $C_2^{\searrow}$ from Section~\ref{sec:lower_bound_jensen} is in fact \emph{not} a true copula.
    The Fréchet copula family $((1-\alpha)\PiC+\alpha\MC)_{\alpha\in[0,1]}$ uniquely traces the upper boundary of both regions, the solid line from $\PiC$ to $C_{\#}$ provides a lower bound for the full attainable region and the dotted line from $\PiC$ to $C_{\#}$ is obtained from the copula family $(C_{\mu})_{\mu\ge 0}$ defined in Section~\ref{sec:two_param_copula}.
    A superclass of the ordinal sums of $\PiC$ attains the lower boundary of the SI region.
    Both regions are convex and closed.
} 
\label{fig:attainable_region}
\end{figure}

Hence, \(\psi(C)\) is a linear functional of \(h\) whilst \(\xi(C)\) is a quadratic functional of \(h\).
This insight is key to derive the results for the $(\xi,\psi)$-region, as it allows us to adapt the approach from \cite{ansari2025exact} and perform constrained optimizations of $\xi$ and $\psi$ over $h$.
The constraints in this optimization are obtained from the following fundamental lemma.

\begin{lemma}[A characterization of copulas, \cite{ansari2025exact}]\label{charSIcop}
    A function \(C\colon [0,1]^2 \to [0,1]\) is a bivariate copula if and only if there exists a family \((h_v)_{v\in [0,1]}\) of measurable functions \(h_v\colon [0,1]\to [0,1]\) such that
    \begin{enumerate}[(i)]
        \item \label{charSIcop1} \(C(u,v) = \int_0^u h_v(t) \de t\) for all \((u,v)\in [0,1]^2\),
        \item \label{charSIcop3}\(\int_0^1 h_v(t) \de t = v\) for all \(v\in [0,1]\),
        \item \label{charSIcop2} \(h_v(t)\) is non-decreasing in \(v\) for all \(t\in [0,1]\).
    \end{enumerate}
\end{lemma}

The rest of the paper is organized as follows.
In Section~\ref{sec:attainable_region}, we establish the upper boundary of the full attainable $(\xi,\psi)$-region via convex optimization over $h$.
Together with the relation \eqref{eq:xi-psi-relation}, we derive in Theorem~\ref{thm:si_region} the exact region for the pair \((\xi(C), \psi(C))\) over the space \(\CC_{\text{SI}}\) consisting of all \emph{stochastically increasing (SI)} copulas (defined in \eqref{eq:si_definition} below).
In Section~\ref{sec:lower_boundary}, we discuss the minimal possible value of $\psi$ for given $\xi$.
If $\xi\ge\frac12$, then this minimal possible value is just $\psi=-\frac12$.
For the case of $\xi\le\frac12$, we give a lower bound based on Jensen's inequality and a convex minimization.
Furthermore, inspired from discrete approximations, we construct a two-parameter copula family that yields particularly small values for $\psi$ when $\xi$ is given.
The key results from both sections are visualized in Figure \ref{fig:attainable_region}.

\section{Upper boundary and SI region}
\label{sec:attainable_region}
The upper boundary for the attainable $(\xi,\psi)$-region is fully characterized by the following theorem.

\begin{theorem}[Maximal value of $\psi$ given $\xi$]\label{thm:exact_region_upper_bound}
For $x\in[0,1]$, it is
\begin{align}\label{eq:exact_region_upper_bound}
    \max\{\psi(C) \mid C \in \CC,~\xi(C) = x\} = \sqrt{x}
\end{align}
and the maximizer is uniquely attained by the Fréchet copula \(C^{\text{Fr}}_\alpha(u,v) := (1-\alpha)uv + \alpha \min(u,v)\) with \(\alpha = \sqrt{x}\).
\end{theorem}

\begin{proof}
Using \eqref{eq:xi-definition} and \eqref{eq:psi-in-terms-of-h}, we formulate and solve an optimization problem that yields the upper boundary \(\psi = \sqrt{\xi}\) as follows:
\begin{align}
    \label{eq:upper_boundary_problem}
    \begin{aligned}
    \text{minimize} \quad &-\int_0^1\int_0^1 \1_{\{t\le v\}} h(t,v) \de t \de v \\
    \text{subject to} \quad
    & 6 \int_0^1\int_0^1 h(t,v)^2 \de t \de v - 2 \le x, \quad \int_0^1 h(t,v) \de t = v \text{ for a.e.~} v \in (0,1)
    \end{aligned}
\end{align}
for given \(x \in (0,1]\) over $h\in L^2((0,1)^2)$.
The first constraint \(6 \int_0^1\int_0^1 h(t,v)^2 \de t \de v - 2 \le x\) ensures that Chatterjee's rank correlation is upper bounded by \(x\).
By Lemma~\ref{charSIcop}, the second constraint is satisfied by all copula partial derivatives.
Thus, provided that the solution to \eqref{eq:upper_boundary_problem} in fact constitutes the first partial derivatives of a copula and satisfies the $\xi$-constraint with equality, it certainly attains the maximum in \eqref{eq:exact_region_upper_bound}.
Note that we made the constraint on $\xi$ an inequality so that the optimization problem is convex, which enables us to use second-order sufficient conditions for optimality, and we thus only need to find a local minimizer.
Also note that the case of $x=0$ is trivially solved by $h(t,v) = v$, which corresponds to the independence copula $\Pi$.

To solve the general minimization problem, we use the framework from the appendix.
More precisely, we let
\[
X=L^2((0,1)^2), \quad Y=\mathbb{R}\times L^2(0,1), \quad
f(h)= -\iint \1_{\{t\le v\}}\,h(t,v)\de t\de v,
\quad K = (-\infty,0]\times\{0\},
\]
and define the constraint operator
\(
  G(h)
  =\Bigl(
      6\iint h^2(t,v)\de t\de v - 2 - x,\;
      \bigl[v\mapsto \int_0^1 h(t,v)\de t - v\bigr]
    \Bigr).
\)
The stationarity condition in Lemma~\ref{lem:KKT} results from setting the Fréchet derivative of the Lagrangian
\begin{equation}\begin{multlined}\label{eq:lagrangian_upper_boundary}
\mathcal{L}(h, (\mu,\gamma)) = -\int_0^1\int_0^1 \1_{\{t\le v\}} h(t,v) \de t \de v \\
+ \mu \left( 6 \iint h^2(t,v) \de t \de v - 2 - x \right) 
+ \int_0^1 \gamma(v) \left( \int_0^1 h(t,v) \de t - v \right) \de v
\end{multlined}\end{equation}
to zero.
Let $k\in L^2((0,1)^2)$ and consider 
$\phi(\varepsilon)=\mathcal{L}(h+\varepsilon k,(\mu,\gamma))$.
Differentiating at $\varepsilon=0$ gives
\[
\mathrm{D}\mathcal{L}(h)[k] 
= \int_0^1\int_0^1 \big(-\1_{\{t\le v\}}+12\mu h(t,v)+\gamma(v)\big)\,k(t,v)\de t\de v.
\]
The KKT conditions from Lemma~\ref{lem:KKT} therefore read as follows:
\begin{enumerate}[(i)]
    \item \emph{Stationarity:} $-\1_{\{t\le v\}} + 12\mu h(t,v) + \gamma(v) = 0$ a.e.\ on $(0,1)^2$.
    \item \emph{Primal Feasibility:} $h$ satisfies the constraints in \eqref{eq:upper_boundary_problem}.
    \item \emph{Dual Feasibility:} $\mu\ge 0$.
    \item \emph{Complementary Slackness:} $\mu\,(6 \iint h^2 - 2 - x)=0$.
\end{enumerate}
We check that the candidate solution
\[
    h^*_v(t) := (1-\sqrt{x})v + \sqrt{x}\1_{\{t \le v\}},
    \quad\mu^* := \frac{1}{12\sqrt{x}},
    \quad\gamma^*(v) := -\frac{1-\sqrt{x}}{\sqrt{x}}v
\]
satisfies all KKT conditions.
First, by construction, \(h^*\) is the first partial derivative of the copula \(C^{\text{Fr}}(u,v;x)\).
A direct calculation shows that \(\xi(C^{\text{Fr}}) = (\sqrt{x})^2 = x\) (see also \cite[Example 4]{fuchs2024quantifying}), so the \(\xi\)-constraint is satisfied and active, consistent with \(\mu^* > 0\); dual feasibility and complementary slackness follow.
Furthermore, the second feasibility condition naturally holds for a copula derivative.
Lastly, it holds that $-\1_{\{t\le v\}} + \gamma^* + 12\mu^* h^* = 0$ since
\[
    -\1_{\{t \le v\}}-\frac{1-\sqrt{x}}{\sqrt{x}}v + 12\left(\frac{1}{12\sqrt{x}}\right)\left((1-\sqrt{x})v + \sqrt{x}\1_{\{t \le v\}}\right)
    = 0,
\]
thereby confirming stationarity.
Hence, all KKT conditions are indeed satisfied.

Next, regarding sufficiency of the KKT conditions, note that the Hessian of the Lagrangian is determined solely by the quadratic \(\xi\)-constraint and satisfies
\begin{align}\label{eq:sosc}
    D_{hh}^2 \mathcal{L}(h^*, (\mu^*,\gamma^*))[k,k]
    = 12\mu^* \|k\|_{L^2}^2
    = \frac{1}{\sqrt{x}} \|k\|_{L^2}^2
\end{align}
for $k\in L^2((0,1)^2)$.
\eqref{eq:sosc} satisfies the quadratic growth condition from Lemma~\ref{lem:BS-3.63}, proving that \(h^*\) is a strict local minimizer.
Since the optimization problem is convex, the local minimizer \(h^*\) is also the unique global minimizer, cf.~\cite[Lem.~4.6]{ansari2025exact}.
The copula \(C^{\text{Fr}}\) generated by \(h^*\) is therefore the unique solution to the optimization problem, and it uniquely achieves the upper boundary \(\psi = \sqrt{\xi}\).
\end{proof}

From Theorem~\ref{thm:exact_region_upper_bound}, it is immediate that $C^{\text{Fr}}_{1/2}$ uniquely maximizes $\psi-\xi$ over all copulas.
Table \ref{tab:footrule_minus_xi_max} compares the maximizers of $\psi-\xi$ for a selection of classical copula families.

\begin{table}[htbp]
  \centering
  \begin{tabular}{lrrrr}
    \toprule
    Family & Parameter & $\xi$ & $\psi$ & $\psi-\xi$ \\
    \midrule
    Clayton        &      1.459 &    0.251 &     0.407 &      0.156 \\
    Frank          &      4.5   &    0.229 &     0.408 &      0.179 \\
    Fréchet        &      0.5   &    0.25  &     0.5  &      0.25  \\
    Gaussian       &      0.614 &    0.225 &     0.397 &      0.172 \\
    Gumbel-Hougaard &      1.781 &    0.249 &     0.424 &      0.174 \\
    Joe            &      1.796 &    0.407 &     0.529 &      0.122 \\
    \bottomrule
  \end{tabular}
  \caption{
    Parameter values that approximately maximize the gap $\psi-\xi$ for the listed copula families together with the corresponding values of $\xi$, $\psi$, and their difference.
    Except for the Fréchet and the Gaussian copula families, where closed-form formulas are available for both rank correlation measures, the entries are obtained by a dense grid search for the parameter and numerical approximations of $\xi$ and $\psi$.
  }
  \label{tab:footrule_minus_xi_max}
\end{table}

As another consequence of this theorem, we obtain the exact $(\xi,\psi)$-region for SI copulas.
A copula \(C\) is said to be \emph{stochastically increasing} (SI) if for all $v\in[0,1]$ the function
\begin{align}\label{eq:si_definition}
u \mapsto \partial_1 C(u,v)
\end{align}
is non-increasing in $u$ outside a Lebesgue-null set.
We denote the class of all SI copulas by \(\CC_{\mathrm{SI}}\), for which a tighter region between \(\xi\) and \(\psi\) is given in Theorem~\ref{thm:si_region}.
For its lower boundary, let $\{(a_k,b_k)\}_{k\in I}$ be a countable family of disjoint intervals in $(0,1)$ and let $\{C_k\}_{k\in I}$ be a family of copulas.
The \emph{ordinal sum} of the $C_k$ with respect to the intervals $(a_k,b_k)$ is the copula $(\langle a_k, b_k \rangle, C_k)_{k\in I}$ defined by
\begin{align}\label{eq:ordinal_sum}
(\langle a_k, b_k \rangle, C_k)_{k\in I}  :=
\begin{cases}
a_k + (b_k - a_k)\;
C_k\!\Bigl(\tfrac{u - a_k}{\,b_k - a_k\,},\;\tfrac{v - a_k}{\,b_k - a_k\,}\Bigr),
& (u,v)\in (a_k,b_k)^2,\;k\in I,\\[6pt]
\MC(u,v),
& \text{otherwise,}
\end{cases}
\end{align}
where $\MC$ is again the upper Fréchet bound.
In particular, if $C_k=\Pi$ for all $k\in I$, then $(\langle a_k, b_k \rangle, \Pi)_{k\in I}$ is called an \emph{ordinal sum of $\Pi$}.
In the following proposition, we characterize the class of SI copulas that satisfy the equality \(\xi = \psi\).

\begin{proposition}[Equality case $\xi=\psi$ in the SI class]\label{prop:si_equality_structure}
Let $C\in\CC_{\mathrm{SI}}$ and set $h_v(t):=\partial_1 C(t,v)$.
The following are equivalent:
\begin{enumerate}
\item[\textnormal{(i)}] $\xi(C)=\psi(C)$.
\item[\textnormal{(ii)}] There exist non-decreasing, measurable functions $A,B:[0,1]\to[0,1]$ and a measurable function $\alpha:[0,1]\to[0,1]$ such that for a.e.\ $v\in[0,1]$, the function $h_v$ is given for a.e.\ $t\in[0,1]$ by
\begin{align}\label{eq:level_two_form}
  h_v(t)=\1_{\{t<A(v)\}}+\alpha(v)\,\1_{(A(v),\,B(v))}(t).
\end{align}
\end{enumerate}
\end{proposition}

\begin{proof}
\emph{(i) $\Rightarrow$ (ii).}
For SI copulas, \cite[Thm.~4.2]{siburg2021stochastic} gives, for every $v\in[0,1]$,
\begin{equation}\label{eq:SI-pointwise}
 (C^\top\!\ast C)(v,v)=\int_0^1 h_v(t)^2\,dt \;\le\; \int_0^v h_v(t)\,dt \;=\; C(v,v).
\end{equation}
By the definitions of $\xi$ and $\psi$,
\[
\psi(C)-\xi(C)
=6\int_0^1 \!\Big(C(v,v)-(C^\top\!\ast C)(v,v)\Big)\,dv
=6\int_0^1\!\left(\int_0^v h_v(t)\,dt-\int_0^1 h_v(t)^2\,dt\right)dv\;\ge 0.
\]
Hence, $\xi(C)=\psi(C)$ forces equality in \eqref{eq:SI-pointwise} for a.e.\ $v$, i.e.~\(
  \int_0^1 h_v(t)^2\,dt=\int_0^v h_v(t)\,dt .
\)
Fix such a $v$.
Since $C(\cdot,v)\le M(\cdot,v)=\min(\cdot,v)$, define
\[
  G_v(t):=\int_0^t \bigl(\1_{[0,v]}(s)-h_v(s)\bigr)\,ds
  \;=\;\min(t,v)-C(t,v)\;\ge0,\qquad t\in[0,1].
\]
Because $t\mapsto h_v(t)$ is a.e.\;non-increasing, modify it on a null set to be non-increasing; then $h_v$ induces a finite nonnegative measure $\mu_v$ via $\de h_v=-\de\mu_v$.
Noticing that $G_v(0)=0$ and $G_v(1)=v-C(1,v)=0$, Lebesgue--Stieltjes integration by parts (see, e.g., \cite[Thm.~21.67 (iv)]{hewitt1975real}) yields
\begin{align}\label{eq:ibp}
  \int_0^1 \bigl(\1_{[0,v]}-h_v\bigr)h_v\,dt
  = \Big[G_vh_v\Big]_{0}^{1} - \int_0^1 G_v\,dh_v
  = \int_0^1 G_v\,d\mu_v \;\ge 0.
\end{align}
Thus, equality in \eqref{eq:SI-pointwise} holds if and only if
\(
  \operatorname{supp}\mu_v \subseteq \{t:\,C(t,v)=M(t,v)\}.
\)
Since $\mu_v$ is the decrease measure of $h_v$, this means $h_v$ has no decrease on the open set $\{C(\cdot,v)<M(\cdot,v)\}$, hence $h_v$ is a.e.\ constant on each connected component of that open set; on its
complement, $h_v=\partial_1 M(\cdot,v)$ equals $1$ (on $\{t<v\}$) or $0$ (on $\{t>v\}$).
As $h_v$ is non-increasing, there can be at most one such component, so there exist $A(v)\le B(v)$ and $\alpha(v)\in[0,1]$ with \eqref{eq:level_two_form}.
Finally, $h_{v_2}(t)\ge h_{v_1}(t)$ a.e.~forces $A,B$ to be non-decreasing in $v$.
Measurability of $A,B$ follows e.g.\ by defining
$A(v):=\sup\{t\le v:\,C(t,v)=t\}$ and $B(v):=\inf\{t\ge v:\,C(t,v)=v\}$, using continuity of $C$.

\smallskip
\emph{(ii) $\Rightarrow$ (i).}
By Lemma~\ref{charSIcop}, it holds $\int_0^1 h_v(t)\,dt=v$, so \eqref{eq:level_two_form} gives
\(
A(v)+\alpha(v)\bigl(B(v)-A(v)\bigr)=v
\)
for a.e.\ $v$.
Hence,
\[
\int_0^1 h_v(t)^2\,dt = A(v)+\alpha(v)^2\bigl(B(v)-A(v)\bigr)
\quad\text{and}\quad
\int_0^v h_v(t)\,dt = A(v)+\alpha(v)\bigl(v-A(v)\bigr),
\]
which are equal because $\alpha(v)\bigl(B(v)-A(v)\bigr)=v-A(v)$.
Therefore, $(C^\top\!\ast C)(v,v)=\int_0^1 h_v^2=C(v,v)$ for a.e.\ $v$, and integrating over $v$ yields
\[
\xi(C)=6\!\int_0^1\!\!\int_0^1 h_v(t)^2\,dt\,dv-2
=6\!\int_0^1\!(C^\top\!\ast C)(v,v)\,dv-2
=6\!\int_0^1\! C(v,v)\,dv-2
=\psi(C).\qedhere
\]
\end{proof}

\begin{remark}
The class in \eqref{eq:level_two_form} contains the ordinal sums of $\Pi$ and \cite[Thm.~5.1]{siburg2021stochastic} implies that ordinal sums of $\Pi$ are in fact the only \emph{symmetric} SI copulas satisfying $\xi=\psi$.
However, further (asymmetric) SI copulas with $\xi=\psi$ do exist, for example letting $A(v)=v/2$ and $B(v)=(v+1)/2$, the function
\[
\partial_1 C(t,v)
= \1_{\{t< v/2\}}
+ v\,\1_{(v/2,\,(v+1)/2)}(t),
\qquad (t,v)\in[0,1]^2,
\]
constitutes the first partial derivative of an SI copula $C$. % and from Proposition~\ref{prop:si_equality_structure} it follows that $\xi(C)=\psi(C)$.
\end{remark}

Let now
\(
\mathcal{R}_{\text{SI}} := \{ (\xi(C), \psi(C)) \mid C \in \CC_{\text{SI}} \}
\)
denote the attainable $(\xi,\psi)$-region for stochastically increasing copulas.
$\mathcal{R}_{\text{SI}}$ is explicitly characterized in the following theorem.

\begin{theorem}[Attainable $(\xi,\psi)$-region for SI Copulas]
    \label{thm:si_region}
It holds that
\[
    \mathcal{R}_{\text{SI}} = \{ (x,y) \in [0,1]^2 \mid x \le y \le \sqrt{x} \}.
\]
The lower bound $x=y$ is attained by ordinal sums of $\Pi$ and the upper bound $y=\sqrt{x}$ uniquely by the Fréchet copula family $(C^{\text{Fr}}_\alpha)_{\alpha \in [0,1]}$.
\end{theorem}

\begin{proof}
The upper boundary $y=\sqrt{x}$ is the same as in Theorem~\ref{thm:exact_region_upper_bound}, since the maximizing Fréchet copulas are all stochastically increasing.
For the lower boundary, note that \((C^\top\ast C)(v,v) \leq C(v,v)\) for all \(v\in[0,1]\) was shown for SI copulas in \cite[Thm.~4.2]{siburg2021stochastic}.
Consequently, as Spearman's footrule is consistent with this pointwise ordering, one has
\begin{align}\label{eq:xi_le_psi}
\xi(C)
= \psi(C^\top \ast C)
\le \psi(C).
\end{align}
By \cite[Thm.~5.1]{siburg2021stochastic}, every ordinal sum of $\Pi$ is \emph{symmetric} and \emph{idempotent}, i.e.\ $C=C^\top$ and $C*C=C$, and therefore attains equality in \eqref{eq:xi_le_psi}.
This traces the full diagonal, as for example the ordinal sums $(\langle 0, \alpha\rangle, \PiC)$ for $\alpha \in [0,1]$ range continuously from $\MC$ to $\PiC$, and hence attains all possible values of $\psi$ between $0$ and $1$.

To show that a point $(x, y)$ in the interior of $\mathcal{R}_{\text{SI}}$ is attained, consider SI copulas $C_0$ and $C_1$ on the border with
\(
    (\xi(C_0), \psi(C_0)) = (y, y) \text{ and } (\xi(C_1), \psi(C_1)) = (y^2, y),
\)
where $C_0$ exists due to the above and $C_1$ due to the observation that the Fréchet copula family achieving the upper bound in Theorem~\ref{thm:exact_region_upper_bound} is SI.
The set of SI copulas is convex, hence $C_\lambda = (1-\lambda)C_0 + \lambda C_1$ is an SI copula for any $\lambda \in [0,1]$.
Furthermore, since $\psi$ is a linear functional of the copula,
\(
\psi(C_\lambda) = y.
\)
Let $f(\lambda) := \xi(C_\lambda)$, then $f(0) = \xi(C_0) = y$ and $f(1) = \xi(C_1) = y^2$.
Writing $h_i:=\partial_1 C_i$ and $h_\lambda:=(1-\lambda)h_0+\lambda h_1$, we obtain the explicit quadratic
\begin{align}\label{eq:xi_convex_continuity}
\xi(C_\lambda)=6\|h_\lambda\|_{L^2}^2-2
=6\bigl((1-\lambda)^2\|h_0\|_{L^2}^2+2\lambda(1-\lambda)\langle h_0,h_1\rangle_{L^2}+\lambda^2\|h_1\|_{L^2}^2\bigr)-2,
\end{align}
which is continuous in $\lambda$, and thus $f$ must take on every value between $y^2$ and $y$.
Since necessarily $y^2 < x < y$, this means that there exists a $\lambda \in (0,1)$ such that $\xi(C_\lambda)=x$, as desired.
\end{proof}

We obtain the following inequality between Chatterjee's rank correlation and Kendall's tau, where the latter is defined as
\(
\tau(C):= 1 - 4 \int_0^1 \int_0^1 \partial_1 C(u,v) \partial_2 C(u,v) \de u \de v.
\)

\begin{corollary}
    If $C\in\CC_{\text{SI}}$, then
    \(
      \xi(C) \le \frac34 \tau(C) + \frac14
    \).
\end{corollary}

\begin{proof}
    By applying \cite[Thm.~4]{bukovvsek2023exact} to $\xi\le\psi$ from Theorem~\ref{thm:si_region} above.
\end{proof}

\begin{remark}
\leavevmode
\begin{enumerate}[(a)]
    \item In \cite{ansari2023dependence}, it was conjectured that $\xi \le \tau$ holds for all SI copulas, but a proof of this stronger inequality is still outstanding.
    \item $\xi$ is actually known to be a continuous function on $\CC_{\text{SI}}$ by \cite[Cor.~3.6]{ansari2025continuity}, but below we use the path-continuity along $C_\lambda$ also for non-SI copulas, hence the explicit argument in \eqref{eq:xi_convex_continuity}.
    \item Recently, in \cite{fuchs2025exact}, the authors studied relationships between multiple measures of association for \emph{lower semilinear (LSL)} copulas.
    In particular, they showed that for LSL copulas it holds $\xi\le\psi$ and that this inequality is sharp.
    Noting that LSL copulas form a convex set and that the Fréchet copula family $(C^{\text{Fr}}_\alpha)_{\alpha\in[0,1]}$ is LSL, it follows that the exact $(\xi,\psi)$-region of LSL copulas $\mathcal{R}_{\text{LSL}}$ satisfies $\mathcal{R}_{\text{LSL}}= \mathcal{R}_{\text{SI}}$.
    \item A copula $C$ is said to be \emph{left tail decreasing (LTD)} if, for every $v \in (0,1]$, the function 
    \[
        u \mapsto \frac{C(u,v)}{u}
    \]
    is non-increasing.
    It is known that SI implies LTD, but the converse does not hold, see \cite[Thm.~5.2.12]{Nelsen-2006}.
    While $\xi \le \psi$ holds for all SI copulas, this inequality fails in the LTD class.
    For example, let
    \[
        \Delta = \tfrac1{12}\begin{pmatrix}
            4 & 0 & 0 \\
            0 & 1 & 3 \\
            0 & 3 & 1
        \end{pmatrix}
    \]
    and consider its associated checkerboard copula \(C^{\Delta}\).
    Explicit calculations then show that
    \(
        \xi(C^{\Delta}) = \frac12 > \frac13 = \psi(C^{\Delta}),
    \)
    see also \cite[Prop.~3.3]{rockel2025measures}.
    \item The exact $(\xi,\psi)$-region for SI copulas does not immediately yield the corresponding region for stochastically decreasing (SD) copulas (where $u \mapsto \partial_1 C(u,v)$ is non-decreasing in $u$ outside a Lebesgue-null set for all $v$), as $\psi$ is \emph{not} a measure of concordance in the sense of \cite[Def.~2.4.7]{Durante-2016}.
    In particular, the SD region would require an exact derivation of the lower boundary for the full attainable $(\xi,\psi)$-region.
\end{enumerate}
\end{remark}

\section{Lower boundary}
\label{sec:lower_boundary}

Deriving the lower boundary of the $(\xi,\psi)$-region is more involved than the upper boundary, and we do not obtain the precise solution to
\begin{align}\label{eq:lower_boundary_problem}
    \min_{C \in \CC} \left( \mu\psi(C) + \xi(C) \right)
\end{align}
for $\mu\ge 0$.
Nevertheless, in Section~\ref{sec:lower_bound_jensen}, we leverage Jensen's inequality to get a more tractable optimization problem, and obtain as a result a lower bound for the minimal possible value in \eqref{eq:lower_boundary_problem}.
Furthermore, we derive the unique copula minimizing Chatterjee's rank correlation amongst all copulas attaining the smallest possible value for Spearman's footrule $\psi=-0.5$.
Lastly, in Section~\ref{sec:two_param_copula}, we introduce a copula family with two parameters that achieves comparably small values for $\mu\psi(C) + \xi(C)$, but it remains unclear if it really attains the true lower boundary of the $(\xi,\psi)$-region.

\subsection{A lower bound by Jensen's inequality}
\label{sec:lower_bound_jensen}

We first define a family of functions $(C^{\searrow}_\mu)_{\mu\ge 0}$ for which we determine Chatterjee's rank correlation and Spearman's footrule explicitly.
We afterwards show that this family indeed provides a lower boundary for the exact $(\xi,\psi)$-region.
For $0\le\mu \le 2$, define the function $C^{\searrow}_\mu : [0,1]^2 \to [0,1]$ via its partial derivative
\begin{align}\label{eq:lower_bound_partial}
    h_\mu(t,v)
    = \partial_1 C^{\searrow}_\mu(t,v)
    := h_1(v)\1_{\{t \le v\}} + h_2(v)\1_{\{t > v\}},
\end{align}
where, with 
\[
    v_0 = \tfrac{\mu}{2+\mu}, 
    \qquad v_1 = \tfrac{2}{2+\mu},
\]
the pair $(h_1(v),h_2(v))$ is given by
\begin{align}\label{eq:lower_bound_h1h2}
    (h_1,h_2)(v) =
    \begin{cases}
        (0,\; \tfrac{v}{1-v}), & v \in [0,v_0], \\[6pt]
        \bigl(v - \tfrac{\mu}{2}(1-v),\; v + \tfrac{\mu}{2}v\bigr), & v \in (v_0,v_1], \\[6pt]
        (2 - \tfrac{1}{v},\; 1), & v \in (v_1,1].
    \end{cases}
\end{align}

Note that $(C^{\searrow}_\mu)_{\mu\ge 0}$ does not form a copula family, because $C^{\searrow}_\mu$ is not generally non-decreasing in its second argument.
Nevertheless, since for this family $(C^{\searrow}_\mu)_{\mu\ge 0}$ the first partial derivatives are well-defined by \eqref{eq:lower_bound_partial}, one can still explicitly calculate the values of $\xi$ and $\psi$ by naturally extending their domains in \eqref{eq:xi-definition} and \eqref{eq:psi-definition}.
The resulting formulas are the content of the next proposition.

\begin{proposition}[Closed-form $\xi$ and $\psi$ for $C^{\searrow}_\mu$]
    \label{prop:compact_formulas}
For every $\mu \in [0,2]$, the functions $C^{\searrow}_\mu$ defined in \eqref{eq:lower_bound_partial} satisfy
\[
    \psi(C^{\searrow}_\mu) = -2v_1^2 + 6v_1 - 5 + \frac{1}{v_1},
    \qquad
    \xi(C^{\searrow}_\mu) = -4v_1^2 + 20v_1 - 17 + \frac{2}{v_1} - \frac{1}{v_1^2} - 12\ln(v_1),
\]
where $v_1 = \tfrac{2}{2+\mu}$ and $0\le\mu\le 2$ as above.
As $\mu$ ranges from $0$ to $2$, the function $\psi(C^{\searrow}_\mu)$ decreases strictly and continuously from $0$ to $-0.5$, whereas $\xi(C^{\searrow}_\mu)$ increases strictly and continuously from $0$ to $12\ln(2)-8\approx0.318$.
\end{proposition}

\begin{proof}
Write $v_0=\tfrac{\mu}{2+\mu}$ and $v_1=\tfrac{2}{2+\mu}$.
These definitions imply the useful identities $v_0+v_1=1$, $1+\tfrac{\mu}{2}=\tfrac{1}{v_1}$, and $\tfrac{\mu}{2}=\tfrac{1-v_1}{v_1}$.
First, we compute $\psi(C^{\searrow}_\mu)=6\int_0^1 v h_1(v)\de v-2$.
From \eqref{eq:lower_bound_h1h2}, the integral splits into two non-zero parts:
\[
\int_0^1 v h_1(v)\,\de v
= \int_{v_0}^{v_1}v\Bigl(v-\tfrac{\mu}{2}(1-v)\Bigr)\,\de v
+ \int_{v_1}^{1}v\Bigl(2-\tfrac{1}{v}\Bigr)\,\de v.
\]
Evaluating these integrals yields
\(
\frac{1}{3}(1+\tfrac{\mu}{2})(v_1^3-v_0^3) - \frac{1}{2}\tfrac{\mu}{2}(v_1^2-v_0^2) + (1-v_1) - (\tfrac{1}{2}-\tfrac{v_1^2}{2}) + (v_1-v_1^2).
\)
Substituting $v_0=1-v_1$ and the identities for $\mu$ simplifies this expression, leading to the final result
\[
\psi(C^{\searrow}_\mu)=-2v_1^2+6v_1-5+\frac{1}{v_1}.
\]
Next, for $\xi(C^{\searrow}_\mu) = 6\int_0^1 \bigl(v h_1(v)^2+(1-v)h_2(v)^2\bigr)\de v-2$, we evaluate the integral over the three intervals from \eqref{eq:lower_bound_h1h2}:
\small
\begin{align*}
&\int_0^1 \bigl(v h_1(v)^2+(1-v)h_2(v)^2\bigr)\de v \\
=& \int_0^{v_0} (1-v)\left(\frac{v}{1-v}\right)^2 \de v
+ \int_{v_0}^{v_1}v(v-\tfrac{\mu}{2}(1-v))^2+(1-v)\left(v+\tfrac{\mu}{2}v\right)^2 \de v
+ \int_{v_1}^{1}v\left(2-\tfrac{1}{v}\right)^2+(1-v) \de v \\
=& \int_0^{v_0} \frac{v^2}{1-v} \de v
+ \int_{v_0}^{v_1} v^2\left(1-\left(\tfrac{\mu}{2}\right)^2\right) + v\left(\tfrac{\mu}{2}\right)^2 \de v
+ \int_{v_1}^{1} 3v - 3 + \frac{1}{v} \de v.
\end{align*}
\normalsize
Integrating these and combining the results leads to
\[
\xi(C^{\searrow}_\mu)
=-4v_1^2+20v_1-17+\frac{2}{v_1}-\frac{1}{v_1^2}-12\ln(v_1).
\]
For monotonicity, we consider the parameter range $\mu \in [0,2]$, which corresponds to $v_1 \in [1, 1/2]$.
Since $v_1$ is a strictly decreasing function of $\mu$, we analyze the signs of the derivatives with respect to $v_1$.
For $\psi$, we have
\[
\frac{\de\psi}{\de v_1}=\frac{-4v_1^3+6v_1^2-1}{v_1^2}.
\]
The numerator is positive for $v_1 \in (1/2, 1]$, so $\frac{\de\psi}{\de v_1} > 0$, which implies that $\psi(C^{\searrow}_\mu)$ is strictly decreasing in $\mu$.
For $\xi$, the derivative is
\[
\frac{\de\xi}{\de v_1}
= \frac{-8v_1^4 + 20v_1^3 - 12v_1^2 - 2v_1 + 2}{v_1^3}
=\frac{-2(v_1-1)(2v_1-1)(2v_1^2-2v_1-1)}{v_1^3}.
\]
On the interval $v_1\in(1/2,1)$, the factors $(v_1-1)$, $(2v_1-1)$, and $(2v_1^2-2v_1-1)$ have signs $(-)$, $(+)$, and $(-)$ respectively.
Their product is positive, making the entire expression for $\frac{\de\xi}{\de v_1}$ negative.
Thus, $\xi(C^{\searrow}_\mu)$ is strictly increasing in $\mu$.
The continuity of the functions is clear and the endpoint values are $\psi(0)=0$, $\psi(2)=-1/2$, $\xi(0)=0$, and $\xi(2)=12\ln(2)-8$.
\end{proof}

As indicated in the beginning of this section, we now turn towards showing that the family $(C^{\searrow}_\mu)_{\mu\in[0,2]}$ indeed provides a lower boundary for the exact $(\xi,\psi)$-region.
Key for this result is the following theorem.

\begin{theorem}[Lower Boundary Estimate]\label{thm:lower_boundary_mu_half}
For $\mu \le 2$, the functional $J:\CC \to \R$ given by
\[
    J(C) \;=\; \mu \psi(C) + \xi(C)
\]
is bounded below by $\mu \psi(C^{\searrow}_\mu) + \xi(C^{\searrow}_\mu)$.
\end{theorem}

\begin{proof}
By Lemma~\ref{charSIcop}, minimizing $J$ over $\CC$ is equivalent to minimizing 
\begin{align}\begin{aligned} \label{eq:jensen-h}
    \text{minimize} \quad & \int_0^1\!\Bigl(\mu \int_0^v h(t,v)\de t + \int_0^1 h(t,v)^2 \de t\Bigr)\de v,\\
    \text{subject to} \quad
    & \int_0^1 h(t,v)\, \de t = v \text{ for a.e.~} v\in(0,1), \\
    & v \mapsto h(t,v)\ \text{is non-decreasing on }[0,1]\ \text{for a.e.~}t, \\
    & 0\le h(t,v) \le 1 \text{ for a.e.~}(t,v)\in[0,1]^2.
\end{aligned}\end{align}
For fixed $v\in(0,1)$ define
\[
h_1(v):=\tfrac1v\int_0^v h(t,v)\de t, 
\qquad
h_2(v):=\tfrac1{1-v}\int_v^1 h(t,v)\de t .
\]
Then
\[
\int_0^1 \mu \int_0^v h(t,v)\de t\de v = \int_0^1 \mu v h_1(v)\de v,
\]
and by Jensen’s inequality with $\phi(x)=x^2$ as in \cite[Frm.~(2.8)]{pecaric1992convex},
\begin{align}
    \label{eq:jensen}
    \begin{aligned}
    \int_0^1 h(t,v)^2 \de t &= \int_0^v h(t,v)^2 \de t + \int_v^1 h(t,v)^2 \de t \\
    &\ge v \left( \frac{1}{v} \int_0^v h(t,v) \de t \right)^2 + (1-v) \left( \frac{1}{1-v} \int_v^1 h(t,v) \de t \right)^2 \\
    &= v h_1(v)^2 + (1-v) h_2(v)^2,
    \end{aligned}
\end{align}
with equality if and only if $h(\cdot,v)$ is a.e.\ constant on $[0,v]$ and on $(v,1]$.
Hence, replacing $h$ by
\begin{align}\label{eq:lower_bound_replacement}
h_{\mu}(t,v):= h_1(v)\,\1_{\{t\le v\}} + h_2(v)\,\1_{\{t>v\}}
\end{align}
does not increase the objective function from \eqref{eq:jensen-h}, so solving
\begin{align}\begin{aligned}\label{eq:jensen-2}
    \text{minimize} \quad & \mu v h_1 + v h_1^2 + (1-v)h_2^2, \\
    \text{subject to} \quad
    & v h_1 + (1-v) h_2 = v, \quad 0\le h_1, h_2 \le 1, 
\end{aligned}\end{align}
for each $v\in(0,1)$ and combining the solutions via \eqref{eq:lower_bound_replacement} gives a lower bound on the functional $J$.
Note that we omitted the monotonicity constraint from \eqref{eq:jensen-h} in \eqref{eq:jensen-2} to obtain a tractable minimization problem, hence the solution to \eqref{eq:jensen-2} only yields a lower bound for $J$ and not necessarily a copula attaining this bound.
Consider the Lagrangian
\begin{multline*}
\mathcal{L}_v((h_1,h_2),(\lambda,\alpha,\beta,\gamma,\delta))
= \mu v h_1 + v h_1^2 + (1-v) h_2^2 \\
- \lambda(v h_1+(1-v)h_2-v) - \alpha h_1 + \gamma(h_1-1) - \beta h_2 + \delta(h_2-1),
\end{multline*}
with multipliers $\lambda\in\R$ and $\alpha,\beta,\gamma,\delta\ge0$.
Analogously to the proof of Theorem~\ref{thm:exact_region_upper_bound} and in accordance with Lemma~\ref{lem:KKT}, the KKT conditions are given as follows:
\begin{enumerate}[(i)]
\item Stationarity: Differentiating $\mathcal{L}_v$ with respect to $h_1$ and $h_2$ gives
\[
\mu v + 2 v h_1 - \lambda v - \alpha + \gamma=0, 
\quad
2(1-v)h_2 - \lambda(1-v) - \beta + \delta=0.
\]
\item Primal feasibility: The constraints in \eqref{eq:jensen-2} holds.
\item Dual feasibility: $\alpha,\beta,\gamma,\delta\ge0$.
\item Complementary slackness: $\alpha h_1=\gamma(1-h_1)=\beta h_2=\delta(1-h_2)=0$.
\end{enumerate}
We distinguish three cases, depending on which bounds for $h_1$ and $h_2$ are active.
\begin{enumerate}[(i)]
\item \emph{Interior solution.}  
If $0<h_1<1$ and $0<h_2<1$, then $\alpha=\beta=\gamma=\delta=0$ and the 
stationarity conditions reduce to
\[
h_1=\tfrac{\lambda-\mu}{2}, \qquad h_2=\tfrac{\lambda}{2}.
\]
Using the linear constraint gives $\lambda=v(2+\mu)$, so
\[
h_1^{\mathrm{int}}(v)=v-\tfrac{\mu}{2}(1-v), 
\qquad 
h_2^{\mathrm{int}}(v)=v+\tfrac{\mu}{2}v.
\]
Feasibility requires $h_1^{\mathrm{int}}\ge 0$ and $h_2^{\mathrm{int}}\le 1$, i.e.
\[
v\ge v_0:=\tfrac{\mu}{2+\mu}, \qquad v\le v_1:=\tfrac{2}{2+\mu}.
\]
Since $\mu\le 2$, we have $v_0\le v_1$, so this case is consistent.

\item \emph{Lower bound active for $h_1$.}  
Suppose $h_1=0$ with $0<h_2<1$.
Then $\gamma=0$ and $\beta=\delta=0$, while
\[
\mu v-\lambda v-\alpha=0, \qquad 2(1-v)h_2-\lambda(1-v)=0.
\]
The constraint gives $h_2=\tfrac{v}{1-v}$ and $\lambda=\tfrac{2v}{1-v}$.
Complementarity requires $\alpha=v(\mu-\lambda)\ge 0$, i.e.\ $v\le v_0$.
Note $h_2\le 1$ if and only if $v\le \tfrac12$, and indeed $v_0\le \tfrac12$ when $\mu\le 2$.
Thus, this case occurs for $v\in[0,v_0]$.

\item \emph{Upper bound active for $h_2$.}  
Suppose $h_2=1$ with $0<h_1<1$.
Then $\beta=0$ and $\alpha=\gamma=0$, while
the constraint yields $h_1=2-\tfrac1v$.
Stationarity gives
\[
\lambda=\mu+2h_1=\mu+4-\tfrac{2}{v}, 
\qquad
\delta=(\lambda-2)(1-v)=(\mu+2-\tfrac{2}{v})(1-v)\ge 0,
\]
equivalent to $v\ge v_1$.
One checks $0\le h_1\le 1$ holds for $v\ge \tfrac12$, 
and indeed $v_1\in[\tfrac12,1]$ for $\mu\le 2$.
Thus, this case occurs for $v\in[v_1,1]$.
\end{enumerate}

Hence, for each possible value of $v\in(0,1)$, we have successfully identified a KKT point.
Collecting cases, the KKT point is given by
\begin{align}\label{eq:jensen_kkt_solution}
(h_1,h_2)(v)=
\begin{cases}
\bigl(0,\ \tfrac{v}{1-v}\bigr), & v\in[0,v_0],\\[6pt]
\bigl(v-\tfrac{\mu}{2}(1-v),\ v+\tfrac{\mu}{2}v\bigr), & v\in(v_0,v_1],\\[6pt]
\bigl(2-\tfrac1v,\ 1\bigr), & v\in(v_1,1],
\end{cases}
\end{align}
with $v_0=\tfrac{\mu}{2+\mu}$, $v_1=\tfrac{2}{2+\mu}$.
The Hessian of $\mathcal{L}_v$ is $\mathrm{diag}(2v,2(1-v))$, which satisfies quadratic growth condition from Lemma~\ref{lem:BS-3.63} since $0<v<1$.
Since \eqref{eq:jensen-2} is another \emph{convex} optimization problem, the KKT point from \eqref{eq:jensen_kkt_solution} is the unique \emph{global} minimizer, cf.~\cite[Lem.~4.6]{ansari2025exact}.
Thus, $h_\mu$ from \eqref{eq:lower_bound_replacement} indeed gives a lower bound to the true minimizer of \eqref{eq:jensen-h}.
Due to \eqref{eq:jensen_kkt_solution} it holds that $h_\mu=\partial_1 C^{\searrow}_\mu$, so it indeed follows
\[
J(C)=\mu\psi(C)+\xi(C)\;\ge\;\mu\psi(C^{\searrow}_\mu)+\xi(C^{\searrow}_\mu),
\]
finishing the proof.
\end{proof}

Together with the upper boundary from Theorem~\ref{thm:exact_region_upper_bound}, the above proposition allows bounding the full attainable $(\xi,\psi)$-region
\begin{align}\label{eq:attainable_region}
\mathcal{R}:=\{(x,y)\in[0, 1] \times [-\tfrac12, 1] \mid (x,y) = (\xi(C), \psi(C)) \text{ for some } C \in \CC\},
\end{align}
which is the content of the following theorem.

\begin{theorem}
    \label{thm:final_explicit_region}
The full attainable $(\xi,\psi)$-region $\mathcal{R}$ is convex, closed and satisfies
\begin{align}\label{eq:region_bound}
\mathcal{R} \subset \left\{ (x,y) \in [0, 1] \times [-\tfrac12, 1] \mid \xi(C^{\searrow}_{\mu(y)}) \le x,~y \le \sqrt{x} \right\},
\end{align}
where the function $\xi(C^{\searrow}_{\mu(y)})$ is explicitly given in Proposition~\ref{prop:compact_formulas}, and the parameter $\mu(y)$ is the unique real solution to the cubic equation $\mu^3 - (4+2y)\mu^2 - (4+8y)\mu - 8y = 0$.
\end{theorem}

Note that $\mu(y)$ can be explicitly computed using Cardano's method, see, e.g., \cite[Chapter 8.8]{lal2017algebra2}.
However, the resulting expression is rather complex and we prefer to omit it.

\begin{proof}
We start with some general observations.
Elementarily, if $C_1,C_2\in\CC$ with $C_1\le C_2$ pointwise, then $\psi(C_1)\le\psi(C_2)$.
Hence, by \cite[Prop.~2.4 \& Cor.~2.5]{ansari2025exact}, for any copula $C$ there exists a stochastically decreasing (SD) copula $C'$ with
$\psi(C')\le \psi(C)$ and $\xi(C')=\xi(C)$.
Consequently, for each fixed $x\in[0,1]$,
\[
\inf\{\psi(C): C\in\CC,\,\xi(C)=x\}
\;=\;
\inf\{\psi(C): C\in\CC_{\mathrm{SD}},\,\xi(C)=x\}\;=:\;\phi(x).
\]
The space of all copulas $\CC$ is compact with respect to the topology of uniform convergence, see, e.g., \cite{trutschnig2017strong}.
The class $\CC_{\mathrm{SD}}$ is a closed subset of $\CC$, and since a closed subset of a compact space is itself compact, it follows that $\CC_{\mathrm{SD}}$ is compact.
Moreover, $\psi$ is continuous on $\CC$, and $\xi$ is continuous on $\CC_{\mathrm{SD}}$ by \cite[Cor.~3.6]{ansari2025continuity}.
Therefore, the map $C\mapsto(\xi(C),\psi(C))$ is continuous on the compact $\CC_{\mathrm{SD}}$, and for every $x$ the lower boundary value
\(
\phi(x)
\)
is attained by some $C_x\in\CC_{\mathrm{SD}}$.
Fix $y\in[-\tfrac12,1]$ and let
\begin{align}\label{eq:horizontal_slice}
\mathcal R_y:=\{x\in[0,1]: (x,y)\in\mathcal R\}=\{\xi(C): C\in\CC,\ \psi(C)=y\}.
\end{align}
If $C_0,C_1\in\CC$ satisfy $\psi(C_0)=\psi(C_1)=y$, then for $C_\lambda:=(1-\lambda)C_0+\lambda C_1$ one has $\psi(C_\lambda)=y$ for all $\lambda\in[0,1]$ by linearity of $\psi$.
Also, $\xi(C_\lambda)$ is continuous in $\lambda$ by \eqref{eq:xi_convex_continuity}, so $\mathcal R_y$ is a (possibly degenerate) closed interval.
Furthermore, this closed interval always contains the point $1$, since for example shuffle-of-min copulas generally satisfy $\xi(C)=1$ and can yield arbitrary values for $\psi$, see, e.g., \cite{durante2012approximation,rockel2025measures}.
Hence, $\mathcal R$ is only limited by the upper and lower boundaries on $\psi$ for each fixed value of $\xi$, and both boundaries are fully attained.

Regarding convexity of $\mathcal R$, the upper boundary $x\mapsto \sqrt{x}$ derived in Theorem~\ref{thm:exact_region_upper_bound} is concave on $[0,1]$.
For the lower boundary, take $x_0,x_1\in[0,1]$ and $C_0,C_1\in\CC$ with $\psi(C_i)=\phi(x_i)$ and $\xi(C_i)=x_i$.
For $\lambda\in[0,1]$ consider $C_\lambda:=(1-\lambda)C_0+\lambda C_1$.
Then it is
\[
    x_\lambda:=\xi(C_\lambda)\le (1-\lambda)x_0+\lambda x_1,
    \qquad
    y_\lambda:=\psi(C_\lambda)=(1-\lambda)\phi(x_0)+\lambda \phi(x_1),
\]
where the inequality follows from convexity of $\xi$ on segments, see \cite[Lem.~2.6]{ansari2025exact}.
Since $\phi$ is non-increasing, the point $(x_\lambda,y_\lambda)$ lies weakly below to the straight line connecting $(x_0,\phi(x_0))$ and $(x_1,\phi(x_1))$ thereby making the lower boundary convex.
A region bounded below by a convex function and above by a concave function is convex.

Regarding closure of $\mathcal R$, note that a finite convex function on an open interval is continuous; hence $\phi$ is continuous on $(0,1)$.
By Proposition~\ref{prop:compact_formulas} and convexity of $\phi$, it must be $\lim_{x\to0}\phi(x)=0=\phi(0)$.
Furthermore, it is easy to construct a copula with $\psi=-0.5$ and $\xi<1$, see for example the checkerboard copula $C_\#$ in Theorem~\ref{thm:checkerboard}.
Therefore, $\phi$ is continuous at $1$ as well.
%In total, $\phi$ is continuous on $[0,1]$.
Recall from the above that $\mathcal{R}$ is only limited by the upper and lower boundaries on $\psi$ for each fixed value of $\xi\in[0,1]$.
The upper boundary is continuous and attained by Theorem~\ref{thm:exact_region_upper_bound}, and the lower boundary is continuous and attained by the considerations above.
Hence, $\mathcal R$ is closed.

Lastly, regarding \eqref{eq:region_bound}, recall again that the upper boundary $y \le \sqrt{x}$ is given by Theorem~\ref{thm:exact_region_upper_bound}.
For the lower boundary, let $C \in \CC$ be any copula with $\psi(C) \le 0$.
By Proposition~\ref{prop:compact_formulas}, there exists a $\mu \in [0,2]$ such that 
\(
  \psi(C) = \psi(C^{\searrow}_\mu).
\)
Then Proposition~\ref{thm:lower_boundary_mu_half} implies
\(
  \xi(C) \;\ge\; \xi(C^{\searrow}_\mu).
\)
Thus, the point $(\xi(C), \psi(C))$ lies weakly to the right of the curve 
\(
\bigl(\xi(C^{\searrow}_\mu), \psi(C^{\searrow}_\mu)\bigr)_{\mu\in[0,2]}.
\)
By Proposition~\ref{prop:compact_formulas}, there exists a unique $\mu(y)\in[0,2]$ such that
\begin{align}\label{eq:mu_y_definition}
   y 
   = \psi(C^{\searrow}_{\mu(y)})
   = -2v_1^2+6v_1-5+\frac{1}{v_1}
\end{align}
with \(v_1 = \frac{2}{2+\mu(y)}\).
The lower–boundary condition $\xi(C) \ge \xi(C^{\searrow}_\mu)$ then becomes
\(
   x \;\ge\; \xi\bigl(C^{\searrow}_{\mu(y)}\bigr).
\)
To express $\mu$ in terms of $y$, a rearrangement of \eqref{eq:mu_y_definition} yields the cubic
\[
\mu^3-(4+2y)\mu^2-(4+8y)\mu-8y=0.
\]
As shown in Proposition~\ref{prop:compact_formulas}, \(\mu\mapsto \psi(C^{\searrow}_\mu)\) is strictly decreasing on \([0,2]\) and ranges from \(0\) to \(-\tfrac12\), so for each \(y\in[-\tfrac12,0]\) there is a unique solution \(\mu(y)\in[0,2]\) of this cubic, and it gives precisely the parameter $\mu(y)$ corresponding to the boundary point with $\psi(C^{\searrow}_{\mu(y)})=y$.
\end{proof}

The lower bound obtained from Jensen's inequality touches the point $(12\ln(2)-8, -0.5)$, namely for $\mu=2$.
For true copulas however, such a small value of Chatterjee's rank correlation cannot be achieved when Spearman's footrule takes its most negative value.
The next theorem shows that for $\psi=-0.5$ the smallest possible value of $\xi$ is actually $0.5$, which is uniquely achieved by a checkerboard copula.

\begin{theorem}[Minimal $\xi$ for $\psi=-0.5$]\label{thm:checkerboard}
Let $C_{\#}$ be the $2\times2$ checkerboard copula with density
\[
c_{\#}(u,v)
= 2\,\Big(\1_{[0,\frac12)}(u)\,\1_{[\frac12,1]}(v)
      +\1_{[\frac12,1]}(u)\,\1_{[0,\frac12)}(v)\Big).
\]
Then $\psi(C_{\#})=-\tfrac12$ and $\xi(C_{\#})=\tfrac12$.
Moreover, for any bivariate copula $C$ with $\psi(C)=-\tfrac12$,
\begin{align}\label{eq:checkerboard_minimality}
\xi(C)\;\ge\;\xi(C_{\#}) \;=\; \tfrac12,
\end{align}
with equality if and only if $C=C_{\#}$ almost everywhere.
\end{theorem}

\begin{proof}
Let $C$ be any bivariate copula with $\psi(C) = -1/2$.
By \cite[Thm.~3.2]{fuchs2019lower}, this is equivalent to the measure of $C$ being supported on the off-diagonal squares of a $2 \times 2$ grid.
Consequently, the checkerboard matrix $\Delta$ associated with any such $C$ is uniquely given by
\[
    \Delta = \begin{pmatrix} 0 & 1/2 \\ 1/2 & 0 \end{pmatrix}.
\]
The checkerboard copula constructed from this matrix, $C_{\Pi}^{\Delta}$, is precisely the copula $C_{\#}$.
By \cite[Thm.~4.1]{rockel2025measures}, we have that $\xi(C_{\Pi}^{\Delta}) \le \xi(C)$.
Thus, $\xi(C_{\#}) \le \xi(C)$ for all copulas $C$ in the considered class, which shows that $C_{\#}$ is the minimizer.
Using \cite[Thm.\;2.2 \& Cor.\;2.3]{ansarifuchs2022asimpleextension}, it is easy to see that equality in \cite[Formula (18) \& Thm.~4.1]{rockel2025measures} holds if and only if $C = C_{\Pi}^{\Delta}$ almost everywhere, which implies the uniqueness claim.
Lastly, a direct calculation shows that $\xi(C_{\#}) = 1/2$.
\end{proof}

\subsection{A novel two-parameter copula family}
\label{sec:two_param_copula}

Whilst we do not precisely obtain the lower boundary of the full attainable $(\xi,\psi)$-region, there is some numerical insight for the shapes of the minimizers that can be obtained from discretized convex optimization using computer software like Python's \texttt{cvxpy} package.
Such a discretized optimization approach allows to numerically recover for example the upper boundary of the full attainable $(\xi,\psi)$-region from Section~\ref{sec:attainable_region}, or the known extremizers for the exact regions studied in \cite{ansari2025exact} and \cite{bukovvsek2024exact}.
In particular, it indicates how the minimizers in the $(\xi, \psi)$ space could look like.
Figure \ref{fig:cvxpy} gives several density plots for different values of $\mu$, discretely minimizing $J(C) := \mu\psi(C) + \xi(C)$ subject to $C \in \mathcal{C}$, or, in other words, discretely solving the optimization problem
\begin{align}
    \label{eq:h_problem}
    \begin{aligned}
    \text{minimize} \quad 
        & 6 \int_0^1 \int_0^1 \Bigl( \mu \, \1_{\{t \le v\}} \, h(t,v) + h(t,v)^2 \Bigr) \, \de t \, \de v \\[0.5em]
    \text{subject to} \quad 
        & 0 \le h(t,v) \le 1 \quad \text{for a.e.~} (t,v) \in (0,1)^2, \\[0.3em]
        & \partial_v h(t,v) \ge 0 \quad \text{in the sense of distributions}, \\[0.3em]
        & \int_0^1 h(t,v) \, \de t = v 
          \quad \text{for a.e.~} v \in (0,1).
    \end{aligned}
\end{align}
The shapes of the copulas in Figure \ref{fig:cvxpy} appear somewhat intuitive, because they tackle two objectives simultaneously to minimize $\mu\psi + \xi$.
Making $\psi$ small is achieved by moving the mass away from the main diagonal to keep $C(u,u)$ as small as possible for $u\in[0,1]$, whilst distributing the mass as evenly as possible between the different conditional distributions $h_v(t)=\partial_1 C(t,v)$ for each $v\in[0,1]$ makes $\xi$ small.

\begin{figure}[ht]
    \centering
    \includegraphics[width=0.32\textwidth]{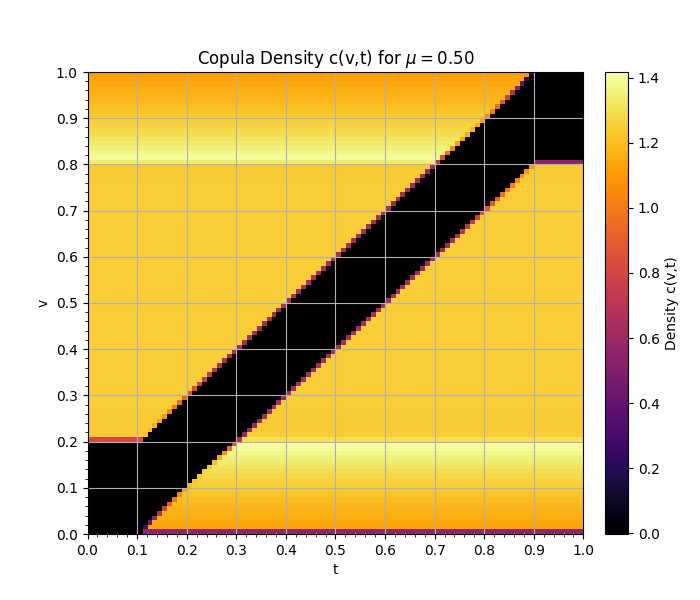}
    \includegraphics[width=0.32\textwidth]{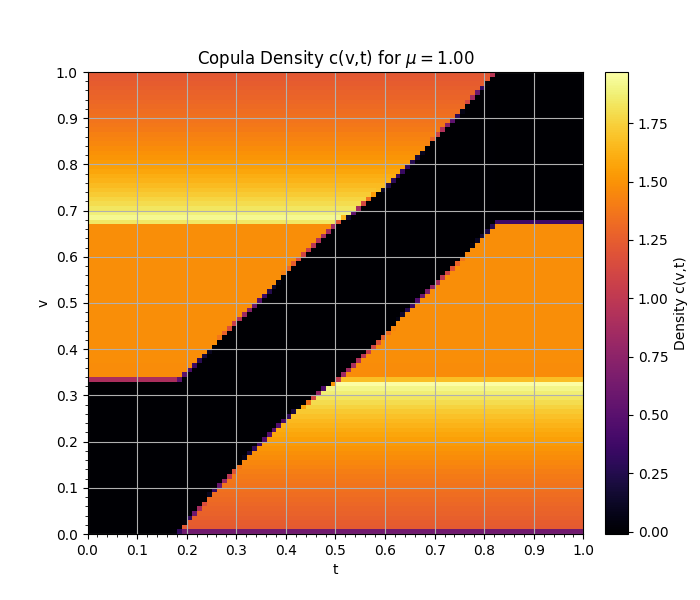}
    \includegraphics[width=0.32\textwidth]{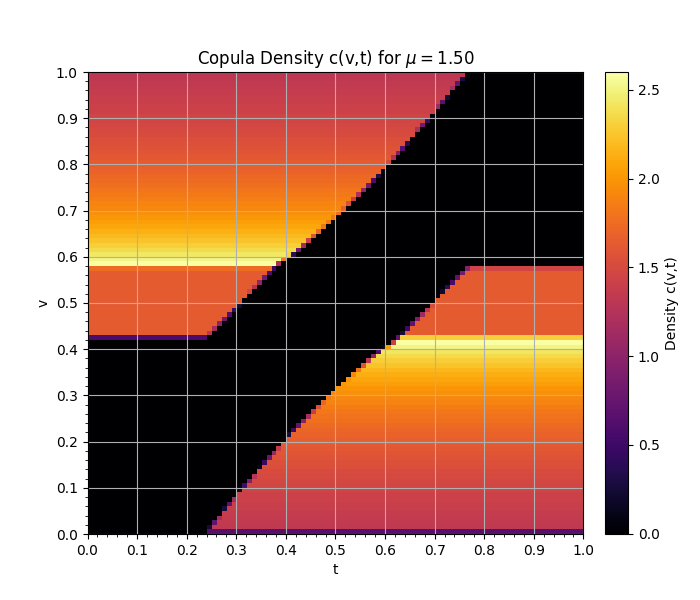} \vspace{-0.5em}\\
    \includegraphics[width=0.32\textwidth]{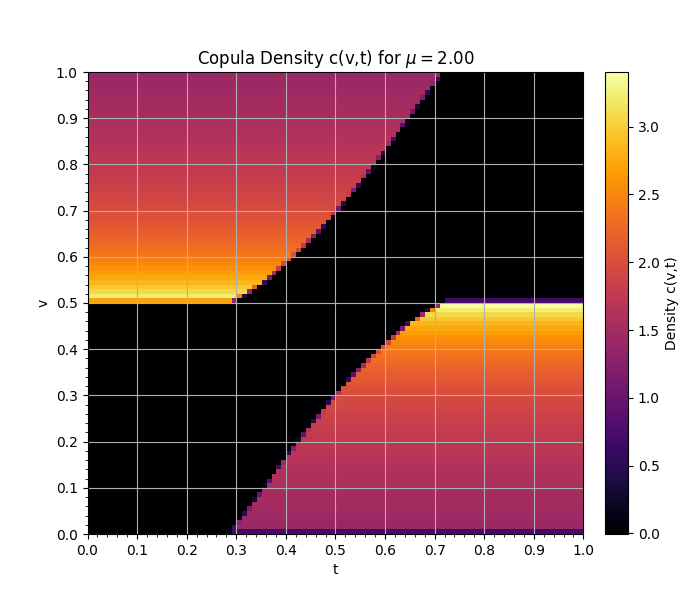}
    \includegraphics[width=0.32\textwidth]{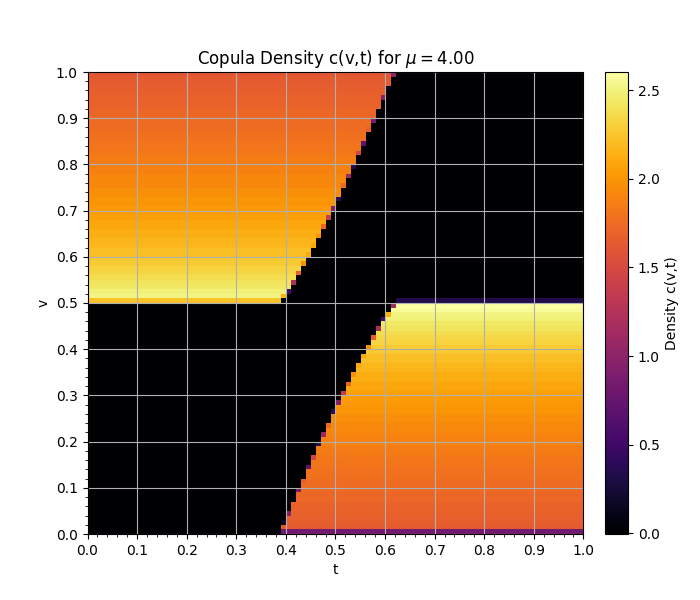}
    \includegraphics[width=0.32\textwidth]{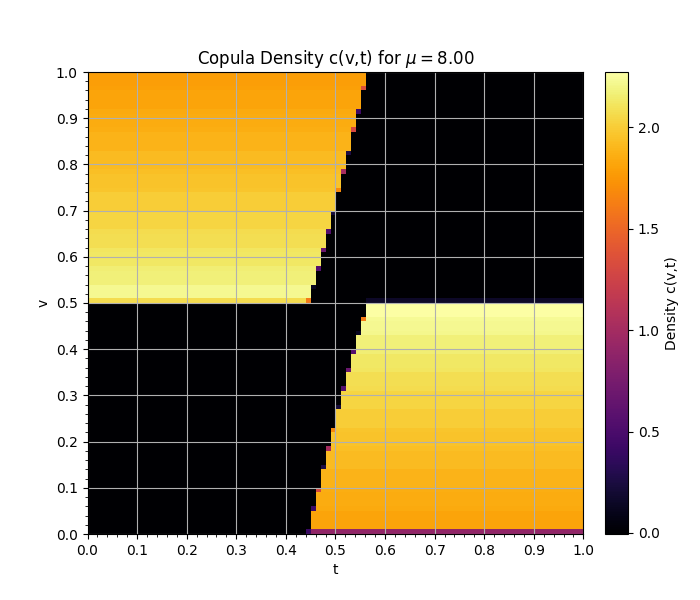} \\
    \caption{Density plots for discrete minimizers of the convex optimization problem \eqref{eq:h_problem} for different values of $\mu$.}
    \label{fig:cvxpy}
\end{figure}

We now define a two-parameter copula family that has a similar structure as the discretely obtained densities from Figure \ref{fig:cvxpy}.
Let $\alpha, \beta \in [0, 0.5)$ be given parameters that determine the shape of a diagonal strip of zero density, denoted by $H_{\alpha,\beta}$.
This strip is defined by a lower boundary function $\psi(s)$ and has a constant vertical width $\beta$.
The lower boundary $\psi : [0,1] \to [0,1-\beta]$ is given piecewise to create rectangular corners of size $\alpha \times \beta$:
\[
    \psi(s) \coloneq
    \begin{cases}
        0, & \text{if } 0 \le s \le \alpha, \\[6pt]
        \dfrac{1-\beta}{1-2\alpha}(s - \alpha), & \text{if } \alpha < s < 1-\alpha, \\[10pt]
        1-\beta, & \text{if } 1-\alpha \le s \le 1.
    \end{cases}
\]
The strip of zero density is then formally defined as
\begin{align}\label{eq:zero_density_region}
    H_{\alpha,\beta}
    \coloneq \left\{ (s,t) \in [0,1]^2 \mid \psi(s) \le t \le \psi(s) + \beta \right\}.
\end{align}
This region $H_{\alpha,\beta}$, bordered by the solid white lines in Figure~\ref{fig:two_param_copula_family}, has Lebesgue measure
\[
    \lambda^2(H_{\alpha,\beta})
    = \int_0^1 \beta \de s
    = \beta,
\]
where $\lambda^2$ denotes the Lebesgue measure on $\R^2$.
To ensure uniform marginal distributions, we introduce a transformation from the copula’s second argument $v$ to an intermediate variable $t$ 
Let
\[
    L(t) \coloneq \int_0^1 \1_{\{\psi(s) \le t \le \psi(s) + \beta\}} \de s
\] 
and define a probability density function $f_T$ on $[0,1]$ by
\begin{equation}\label{eq:ft_density}
    f_T(t) \coloneq \frac{1 - L(t)}{1 - \beta}.
\end{equation}
Let $F_T(v) = \int_0^v f_T(x)\, dx$ be the corresponding CDF with quantile function $F_T^{-1}$.
We then define the density
\begin{align}\label{eq:two_param_copula_density}
    c_{\alpha,\beta}(u,v)
    = \frac{1}{(1 - \beta)\, f_T\left(F_T^{-1}(v)\right)}
      \, \1_{\left\{(u, F_T^{-1}(v)) \notin H_{\alpha,\beta}\right\}},
\end{align}
on $[0,1]^2$, and denote by $C_{\alpha,\beta}$ the copula associated with this density.
A visualization of $H_{\alpha,\beta}$ and $c_{\alpha,\beta}$ is provided in Figure~\ref{fig:two_param_copula_family}.

\begin{figure}[ht]
    \centering
    \includegraphics[width=0.49\textwidth]{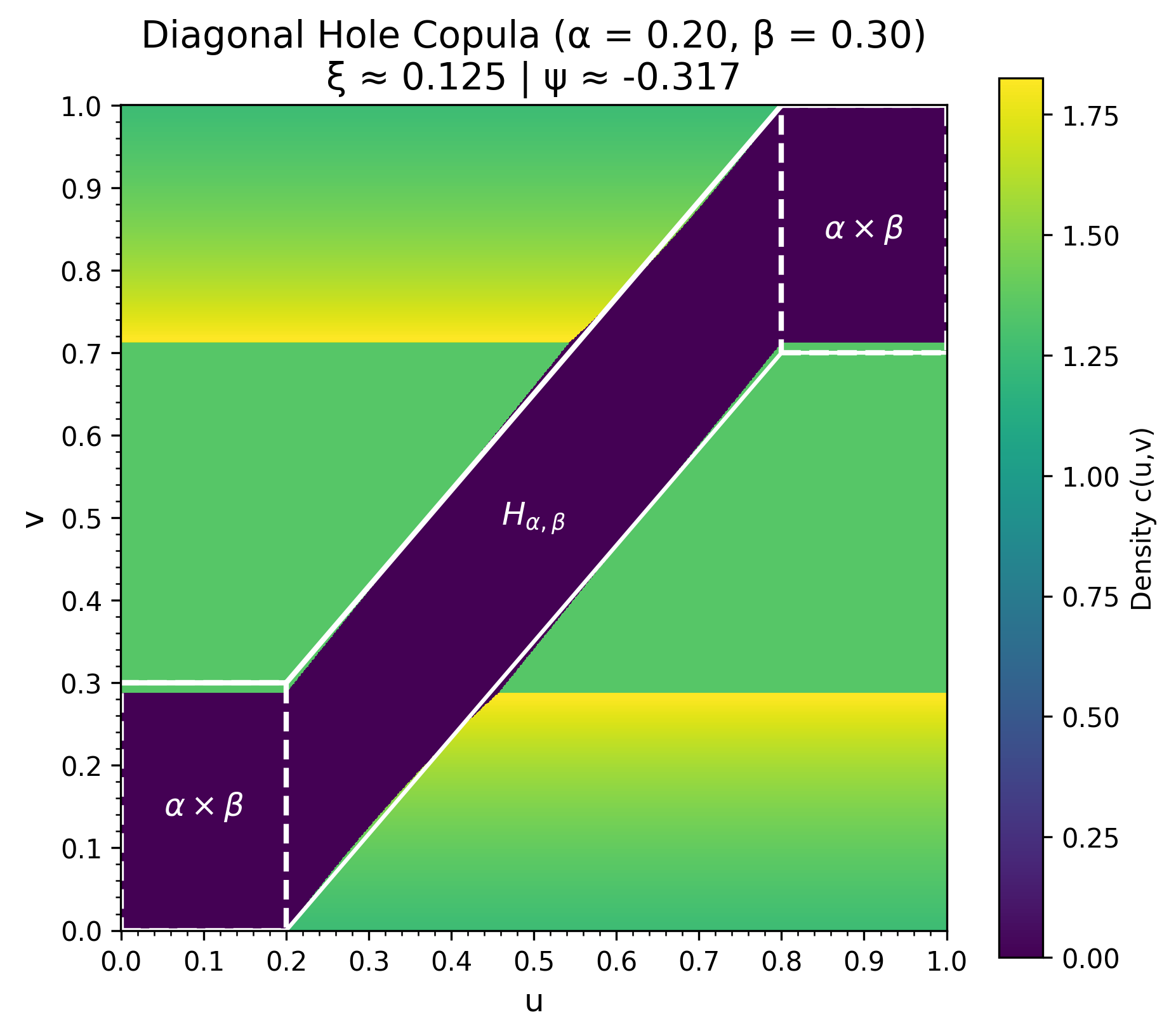}
    \includegraphics[width=0.49\textwidth]{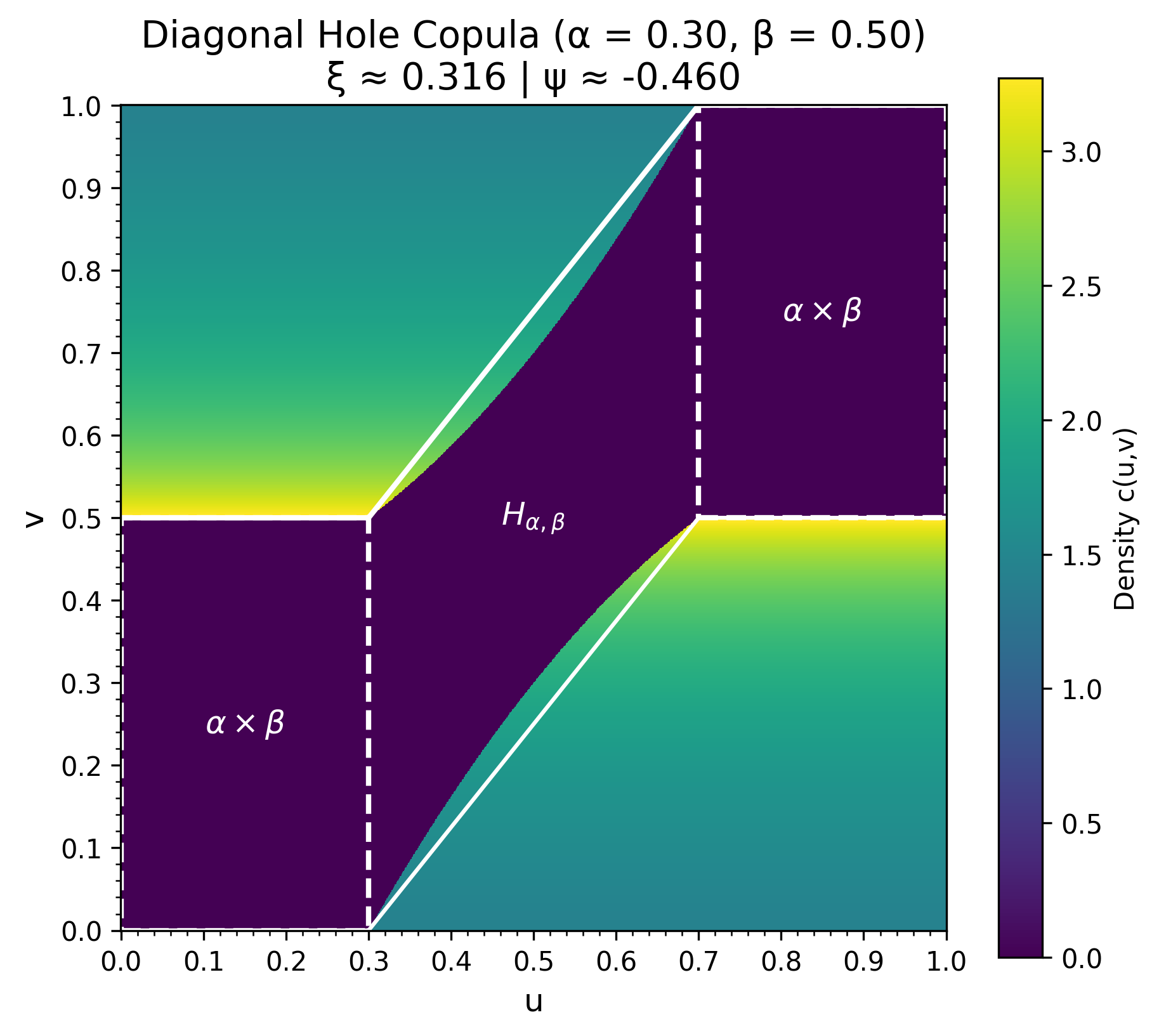}
    \caption{
        A visualization of $C_{\alpha, \beta}$ for $(\alpha, \beta) = (0.2, 0.3)$ (left) and $(\alpha, \beta) = (0.3, 0.5)$ (right).
        The parameters describe the shapes of rectangles in the lower-left and upper-right corners, which together with a connecting band form the shape $H_{\alpha, \beta}$ where the copula should have a density of zero.
        The parameter transform bends this zero-density area to ensure the uniformity of the marginals, hence $H_{\alpha, \beta}$ does not exactly match the zero-density area of $C_{\alpha, \beta}$.
    }
    \label{fig:two_param_copula_family}
\end{figure}

\begin{proposition}\label{prop:two_param_copula_is_copula}
    $C_{\alpha,\beta}$ is indeed a copula for all $\alpha,\beta\in[0,0.5]$.
\end{proposition}

\begin{proof}
We prove that for given $\alpha,\beta\in[0,1]$ the function $c_{\alpha, \beta}$ is a valid copula density.
Non-negativity is clear from the definition of $c_{\alpha, \beta}$ in \eqref{eq:two_param_copula_density}, so we only need to verify that $c_{\alpha, \beta}$ has uniform marginals.
Regarding the marginals, we use the change of variables $t = F_T^{-1}(v)$, which implies $v = F_T(t)$ and $\de v = f_T(t) \de t$.
For the first marginal distribution, we integrate with respect to $v$ for arbitrary $u\in[0,1]$:
\[
    \int_0^1 c_{\alpha, \beta}(u, v) \de v = \int_0^1 \left( \frac{1}{(1-\beta)f_T(t)} \1_{\{(u, t) \notin H_{\alpha, \beta}\}} \right) f_T(t) \de t 
    = \frac{1}{1-\beta} \int_0^1 \1_{\{(u, t) \notin H_{\alpha, \beta}\}} \de t 
    = 1.
\]
Regarding the second marginal, let $v\in[0,1]$.
Let $t = F_T^{-1}(v)$ be fixed for the integration.
\begin{align*}
    \int_0^1 c_{\alpha, \beta}(u, v) \de u &= \int_0^1 \frac{1}{(1 - \beta)f_T(t)} \1_{\{(u, t) \notin H_{\alpha, \beta}\}} \de u \\
    &= \frac{1}{(1-\beta)f_T(t)} \int_0^1 1 - \1_{\left\{(u, t) \in H_{\alpha, \beta}\right\}} \de u \\
    &= \frac{1}{(1-\beta)f_T(t)} \left( 1 - \int_0^1 \1_{\left\{\psi(u) \le t \le \psi(u) + \beta\right\}} \de u \right) \\
    &= \frac{1 - L(t)}{(1-\beta)f_T(t)} = 1,
\end{align*}
where the last equality is due to \eqref{eq:ft_density}.
Since $c_{\alpha, \beta}(u, v)$ is non-negative and has uniform marginals, it is a valid copula density.
\end{proof}

Whilst the exact relationship between $\alpha(\mu)$ and $\beta(\mu)$ that yields the lowest possible values for $J(\mu):=\mu\psi +\xi$ within this two-parameter copula family remains undetermined, Figure \ref{fig:cvxpy} suggests that for the first regime ($\mu \le 2$), a linear relationship of the form $\alpha(\mu) = \frac35\beta(\mu)$ yields a good choice to get small values for $J(\mu)$.
We hence define a concrete, continuous path through the two-dimensional parameter space $(\alpha, \beta)$ by
\begin{equation}\label{eq:dashed_line_copula_family}
\alpha(\mu)= \frac3{20}\mu \1_{\left\{\mu \le 2\right\}} + \left(\frac12 - \frac{2}{5\mu}\right) \1_{\{\mu > 2\}}, 
\qquad \beta(\mu)= \frac14\min\left\{\mu, 2\right\},
\end{equation}
and define the one-parameter copula family
\(
C_\mu := C_{\alpha(\mu), \beta(\mu)}
\)
for $\mu\ge 0$.
In words, $(\alpha(\mu), \beta(\mu))$ starts at $(0,0)$, follows the line $\alpha = \frac35 \beta$ until it hits the upper boundary $\beta = 0.5$, and then continues along this upper boundary until it reaches the point $(0.5, 0.5)$.
This path roughly follows the shapes of the densities from Figure \ref{fig:cvxpy}.
Numerically evaluating $(\xi(C_\mu),\psi(C_\mu))$ along this path yields the dashed line from Figure \ref{fig:attainable_region}.
This most likely does not yield the true minimizing curve, since alone the parameter choice of $\alpha=\frac{3}{5}\beta$ is only roughly numerically motivated from Figure \ref{fig:cvxpy}.
Nevertheless, it does come reasonably close to the (unattainable) lower bound from Section~\ref{sec:lower_bound_jensen}.
Furthermore, the minimizing $\mu$ for the term $\xi(C_\mu) + \psi(C_\mu)$ lies around $1.26$, which gives a comparably small value of $\xi+\psi$ when comparing it with other classical copula families, see Table \ref{tab:footrule_plus_xi_min}.

\begin{table}[ht!]
  \centering
  \begin{tabular}{lrrrr}
    \toprule
    Family & Parameter & $\xi$ & $\psi$ & $\xi+\psi$ \\
    \midrule
    $(C^{\searrow}_\mu)_{\mu\ge0}$        &  1.646 & 0.251 & $-0.488$ &    $-0.237$ \\
    $(C_\mu)_{\mu\ge0}$        &  1.265 & 0.135 & $-0.328$ &    $-0.193$ \\
    Clayton        &     $-0.384$ &      0.089 &     $-0.188$ &     $-0.099$ \\
    Frank          &     $-2.789$ &      0.110 &     $-0.240$ &     $-0.131$ \\
    Gaussian       &     $-0.404$ &      0.090 &     $-0.211$ &     $-0.121$ \\
    Gumbel-Hougaard&      1 &      0 &      0 &      0 \\
    Joe            &      1 &      0 &      0 &      0 \\
    Lower Fréchet  &  0.25 & $0.0625$ & $-0.125$ &    $-0.0625$ \\
    \bottomrule
  \end{tabular}
  \caption{
  Parameter values that approximately minimize the sum $\xi+\psi$ for common copula families together with the corresponding values of $\xi$, $\psi$, and their sum.
  The entries for the Clayton and Frank copula families are obtained by a dense grid search for the parameter and numerical approximations of $\xi$ and $\psi$.
  The lower Fréchet copula family is here defined as $\lambda \Pi + (1-\lambda) W$ for $\lambda\in[0,1]$, which is the stochastically decreasing rearrangement of the Fréchet upper boundary from Section~\ref{sec:attainable_region}.
  }
  \label{tab:footrule_plus_xi_min}
\end{table}

\appendix
\section{Appendix}

In this appendix, we provide a brief overview of the Banach space optimization framework from \cite[Ch.~3]{bonnans2013perturbation} used in the proofs of Theorem~\ref{thm:exact_region_upper_bound} and Theorem~\ref{thm:lower_boundary_mu_half}, and adapt it to our setting.
The general setting is
\[
  \min_{x\in X} f(x)
  \quad\text{subject to}\quad
  G(x)\in K,
\]
where $f:X\to\R$ and $G:X\to Y$ are twice continuously Fréchet differentiable, $X$ and $Y$ are Banach spaces, and $K\subseteq Y$ is a non-empty closed convex set.
A point $x_0\in X$ is called \emph{feasible} if $G(x_0)\in K$.
The \emph{generalized Lagrangian} is
\[
  \mathcal{L}(x,a,\lambda)\ :=\ a\,f(x)\;+\;\langle \lambda, G(x)\rangle,
  \qquad (x,a,\lambda)\in X\times \R \times Y^*,
\]
where $Y^*$ denotes the dual space of $Y$.
Following the definition in \cite[Eq.~(3.21)]{bonnans2013perturbation}, a pair $(a,\lambda)\in \R\times Y^*$ is called a \emph{generalized Lagrange multiplier} at a feasible $x_0\in X$ if it satisfies the following conditions:
\begin{equation}\label{eq:BS-condition}
  D_x \mathcal{L}(x_0,a,\lambda)=0,\qquad
  \lambda \in N_K(G(x_0)),\qquad
  a\ge 0,\quad (a,\lambda)\neq (0,0).
\end{equation}
Here, the \emph{normal cone} of $K$ at $y\in K$ is
\begin{align}\label{eq:normal_cone}
  N_K(y)\ :=\ \bigl\{\lambda\in Y^*: \langle \lambda, z-y\rangle \le 0\ \ \forall z\in K\bigr\}.
\end{align}
To connect this with the familiar Karush-Kuhn-Tucker conditions, we assume that the constraint set $K$ splits into equality and inequality components.
Concretely, let
\(
  Y \;=\; Y_= \times \Big(\prod_{j=1}^m Y_j\Big),
\)
where $Y_=$ is a Banach space for equality constraints and each $Y_j$ is a Banach space that corresponds to an inequality block.
We then take
\(
  K \;=\; \{0_{Y_=}\}\times \Big(\prod_{j=1}^m K_j\Big),
\)
where each inequality cone is
\(
  K_j \;=\; \{y\in Y_j:\; y \le 0\}.
\)
Thus, feasibility requires $G_=(x_0)=0$ and $G_j(x_0)\le 0$ for each $j=1,\dots,m$.

\begin{lemma}[Equivalence of multipliers and KKT]\label{lem:KKT}
Let $x_0$ be a feasible point.
A pair $(a, \lambda) \in \mathbb{R} \times Y^*$ with $a>0$ is a generalized Lagrange multiplier for $x_0$ if and only if, after normalizing $a$ to $1$, the components of $\lambda = (\gamma, \lambda_1, \dots, \lambda_m)$ satisfy the following four Karush–Kuhn–Tucker (KKT) conditions:
\begin{enumerate}[(i)]
\item \emph{Stationarity:}
\(
D_x \mathcal{L}(x_0, 1, (\gamma, \lambda_1, \dots, \lambda_m)) = 0 \text{ in } X^*.
\)
\item \emph{Primal feasibility:}
\(
 G_=(x_0)=0,\quad G_j(x_0)\in K_j\ \ (j=1,\dots,m).
\)
\item \emph{Dual feasibility:}
\(
 \lambda_j \ge 0.
\)
\item \emph{Complementarity:}
\(
 \langle \lambda_j,\, G_j(x_0)\rangle=0 \qquad (j=1,\dots,m).
\)
\end{enumerate}
\end{lemma}

\begin{proof}
\enquote{\emph{$\Rightarrow$}}.
Assume that $(a,\lambda)$ is a generalized Lagrange multiplier with $a>0$ as in \eqref{eq:BS-condition}.
Normalizing by $a$ we may take $a=1$.
Then $D_x \mathcal{L}(x_0,1,\lambda)=0$, and since
\(
\mathcal{L}(x,1,\lambda)=f(x)+\langle \lambda,G(x)\rangle
\),
decomposing $\lambda=(\gamma,\lambda_1,\dots,\lambda_m)$ and $G=(G_=,G_1,\dots,G_m)$ yields the stationarity condition (i).

\smallskip
The second condition in \eqref{eq:BS-condition} is $\lambda \in N_K(G(x_0))$, where the normal cone is defined in \eqref{eq:normal_cone} above as
\[
N_K(y)\ :=\ \{\lambda\in Y^*: \langle \lambda, z-y\rangle \le 0 \ \ \forall z\in K\}.
\]
Since $K=\{0\}\times\prod_{j=1}^m K_j$ is a product, the cone factorizes as
\[
N_K(G(x_0)) = Y_=^* \times \prod_{j=1}^m N_{K_j}(G_j(x_0)).
\]
For each inequality block $K_j=\{y\in Y_j: y\le 0\}$ and $y_j=G_j(x_0)\in K_j$, the definition gives
\[
\lambda_j \in N_{K_j}(y_j)
\;\;\Longleftrightarrow\;\;
\langle \lambda_j, z-y_j\rangle \le 0 \quad \forall z \le 0.
\]
Choosing $z=y_j-tw$ with $w\ge 0$ and $t>0$ shows that $\lambda_j\ge 0$.
Choosing $z=0$ yields $\langle \lambda_j, y_j\rangle \ge 0$, while choosing $z=2y_j$ yields
$\langle \lambda_j, y_j\rangle \le 0$, hence $\langle \lambda_j, y_j\rangle=0$.
Thus,
\[
N_{K_j}(G_j(x_0))=\{\lambda_j\in Y_j^*:\ \lambda_j\ge 0,\;\langle \lambda_j,G_j(x_0)\rangle=0\}.
\]
It follows that $\lambda \in N_K(G(x_0))$ is equivalent to primal feasibility, dual feasibility, and complementarity, i.e., conditions (ii)–(iv).

\smallskip
\enquote{\emph{$\Leftarrow$}}.
If (i)–(iv) hold for some $(\gamma,\lambda_1,\dots,\lambda_m)$, then (i) gives
$D_x \mathcal{L}(x_0,1,\lambda)=0$, and (ii)–(iv) are exactly the condition $\lambda \in N_K(G(x_0))$ established above.
Therefore, $(1,\lambda)$ is a generalized Lagrange multiplier for $x_0$.
\end{proof}

A point $x_0\in X$ is called a \emph{Karush–Kuhn–Tucker (KKT) point} if there exists a \emph{KKT multiplier} $(\gamma,\lambda)$ such that the conditions (i)--(iv) from Lemma~\ref{lem:KKT} hold.
Let \(\Lambda(x_0)\) denote the set of all KKT multipliers for \(x_0\).

\begin{lemma}[Quadratic growth]\label{lem:BS-3.63}
 Suppose that there exists a constant \(\beta>0\) such that
 \begin{align}\label{eq:quadratic_growth}
  \sup_{(\gamma,\lambda)\in\Lambda(x_0)}
   D_{xx}^{2}\mathcal{L}(x_0,1,(\gamma,\lambda))(k,k)
   \;\ge\;
   \beta\|k\|^{2}
   \quad
   \forall k\in X
 \end{align}
 holds at a feasible point \(x_0\).
 Then there exist a constant $C > 0$ and a neighborhood $N_0$ of $x_0$ such that for all feasible $x \in N_0$ it holds that
 \(
 f(x) \ge f(x_0) + C\|x - x_0\|^2
 \).
\end{lemma}

\begin{proof}
 The quadratic growth condition \eqref{eq:quadratic_growth} requires the set of KKT multipliers $\Lambda(x_0)$ to be non-empty and imposes a strict positivity condition on the Hessian of the Lagrangian.
 By Lemma~\ref{lem:KKT}, this is equivalent to the existence of a generalized Lagrange multiplier $(a,\lambda)$ with $a>0$ satisfying \eqref{eq:BS-condition}.
 The result thus follows from \cite[Thm.~3.63 (i)]{bonnans2013perturbation}.
\end{proof}

Note that condition \eqref{eq:quadratic_growth} is a slightly stronger form of the
second-order sufficient condition in \cite[Thm.~3.63 (i)]{bonnans2013perturbation},
but it is sufficient for our applications.

\bibliographystyle{plainnat}
\bibliography{Literature}
\end{document}